%% file: FlexCone.tex
\documentclass[11pt]{amsart}

\usepackage{amsmath,amssymb,amscd,amsthm}

\usepackage[dvips]{graphicx}
\usepackage{epsfig}
\usepackage{graphics}
\usepackage{color}
\graphicspath{{Fig/}}

\usepackage{url}

\newtheorem{dfn}{Definition}[section]
\newtheorem{thm}[dfn]{Theorem}
\newtheorem{prp}[dfn]{Proposition}
\newtheorem{lem}[dfn]{Lemma}
\newtheorem{cor}[dfn]{Corollary}

\newtheorem{Alphathm}{Theorem}

\theoremstyle{definition}
\newtheorem{rem}[dfn]{Remark}
\newtheorem{exl}[dfn]{Example}

\theoremstyle{plain}

\def\co{\colon\thinspace}

\def\R{{\mathbb R}}
\def\C{{\mathbb C}}
\def\Z{{\mathbb Z}}
\def\Sph{{\mathbb S}}

\def\Disk{{\mathbb D}}
\def\tr{\mathrm {tr}}
\def\id{\mathrm {id}}
\def\dev{\mathrm{dev}}
\def\sl{\mathrm{sl}}
\def\Ad{\mathrm{Ad}}

\def\H{{\mathbb H}}
\def\dS{{\mathrm d}{\mathbb{S}}}

\def\phi{\varphi}
\def\epsilon{\varepsilon}

\def\cV{{\mathcal V}}
\def\cM{{\mathcal M}}

\def\dist{\mathrm{dist\,}}

\def\Iso{\operatorname{Iso}}

\DeclareMathOperator{\arctanh}{arctanh}

\title[Infinitesimally flexible cone-manifolds]{Examples of infinitesimally flexible 3--dimensional hyperbolic cone-manifolds}
\author{Ivan Izmestiev}
\thanks{Supported by the DFG Research Unit 565 ``Polyhedral Surfaces''}
\date{September 25, 2009}

\begin{document}
\maketitle

\begin{abstract}
Weiss and, independently, Mazzeo and Montcouquiol recently proved that a 3--dimensional hyperbolic cone-manifold (possibly with vertices) with all cone angles less than $2\pi$ is infinitesimally rigid. On the other hand, Casson provided 1998 an example of an infinitesimally flexible cone-manifold with some of the cone angles larger than $2\pi$.

In this paper several new examples of infinitesimally flexible cone-manifolds are constructed. The basic idea is that the double of an infinitesimally flexible polyhedron is an infinitesimally flexible cone-manifold. With some additional effort, we are able to construct infinitesimally flexible cone-manifolds without vertices and with all cone angles larger than $2\pi$.
\end{abstract}

\section{Introduction}
\subsection{Hyperbolic cone-manifolds}
\label{subsec:ConeMan}
A hyperbolic cone-manifold $M$ is a manifold with metric structure that is hyperbolic away from a codimension 2 subcomplex $\Sigma$ and exhibits cone-like singularities at the points of $\Sigma$. In this paper, only 3--dimensional hyperbolic cone-manifolds are studied. Thus, the singular locus $\Sigma$ is a graph.

By definition, at a point $x \in \Sigma$, a cone-manifold is locally isometric to a hyperbolic cone over a spherical cone-surface $L_x$ homeomorphic to the sphere, see \cite{BLP05}. The cone-surface $L_x$ is called the link of $x$. If $L_x$ has exactly two cone-points, then $x$ is said to lie on an edge of $\Sigma$; if $L_x$ has more than two cone-points, then $x$ is called a vertex of $\Sigma$.

Sometimes, when speaking about cone-manifolds, one means only those whose singular locus is a 2-codimensional submanifold, that is a disjoint union of closed curves in our case. We will use the term \emph{cone-manifolds without vertices} to describe this situation; general cone-manifolds will be sometimes called \emph{cone-manifolds with vertices}.

An alternative way of viewing cone manifolds is as follows. Take a collection of hyperbolic polyhedra and glue them isometrically face-to-face so that the resulting space $K$ is a homeomorphic to a 3--manifold. Then $K$ is clearly a hyperbolic cone-manifold. All examples of cone-manifolds in this paper will be of this kind.

Conversely, it is plausible that every cone-manifold has a geodesic triangulation (i.~e. can be glued from simplices). But we have not found a proof of this in the literature.

\subsection{Rigidity theorems for cone-manifolds}
\label{subsec:RigThms}
Compact hyperbolic manifolds of dimension 3 and higher are known to be globally rigid (Mostow rigidity theorem) and infinitesimally rigid (Calabi-Weil rigidity theorem). By constrast, it is easy to find a smooth family of hyperbolic cone-manifolds, say, by changing the shapes of polyhedra in the last paragraph of Section~\ref{subsec:ConeMan}. As a consequence, cone-manifolds are neither globally nor infinitesimally rigid in regard to deformations that preserve only the topology of the manifold and of its singular locus.

One can reduce degrees of freedom by requiring all cone angles at the singular locus to be constant during the deformation. This leads to remarkable rigidity results. Hodgson and Kerckhoff \cite{HK98} showed that, under this restriction, 3--dimensional hyperbolic cone-manifolds without vertices are infinitesimally rigid if all cone angles are less than $2\pi$. Weiss \cite{Wei05} proved the same by allowing vertices and restricting all cone angles to be less or equal $\pi$. And recently, Weiss \cite{Wei09} and Mazzeo and Montcouquiol \cite{MM09} extended this to cone-manifolds with vertices and cone-angles less than $2\pi$.

Infinitesimal rigidity can be used to prove local rigidity, which means that cone-manifolds (with restrictions on the singular locus mentioned above) are locally parametrized by their cone angles, \cite{HK98, Mont09, Wei09}.
Hodgson and Kerckhoff \cite{HK98} suggested how global rigidity can be derived from local rigidity, if one is able to extend ``small'' deformations to a ``big'' one that makes all cone angles vanish at the end. Kojima \cite{Koj98a} accomplished this, under assumption that all cone angles are less or equal $\pi$ and that there are no vertices. This was extended by Weiss \cite{Wei07}, to include the case of cone-manifold with vertices. It is an open question whether cone-manifolds with cone angles less than $2\pi$ are globally rigid.

Without any restrictions on the values of cone angles, Hodgson and Kerckhoff \cite{HK08} proved infinitesimal and local rigidity for cone-manifolds without vertices provided that the singular locus has a tubular neighborhood of radius at least $\arctanh \frac{1}{\sqrt{3}}$.

Shortly after the first rigidity result appeared, Casson \cite{Cas98} presented an example of an \emph{infinitesimally flexible} cone-manifold with vertices and with some cone angles bigger than $2\pi$. This coexistence of rigidity theorems and flexibility examples raised the question where the border between them lies. For example, it was asked if all cone-manifolds with cone angles larger than $2\pi$ are infinitesimally rigid.

\subsection{Results of the present paper}
In this paper we present a series of examples of infinitesimally flexible cone-manifolds, inferring that the result of Weiss \cite{Wei09} and Mazzeo and Montcouquiol \cite{MM09} is the best possible, in some sense. Besides, we show that global rigidity fails if cone angles larger that $2\pi$ are allowed.

The main idea is that doubling an infinitesimally flexible polyhedron $P$ produces an infinitesimally flexible cone-manifold. An infinitesimal deformation of the double is constructed by choosing an infinitesimal isometric deformation on one copy of $P$ and the opposite deformation on the other copy. As the variations of dihedral angles cancel each other, the cone angles of the double are stable.

In order to construct examples of infinitesimally flexible hyperbolic polyhedra, we use an elegant theorem of Pogorelov: a hyperbolic polyhedron is infinitesimally flexible iff its image in a Klein model is infinitesimally flexible as a Euclidean polyhedron. Several examples of infinitesimally flexible Euclidean polyhedra are known, the simplest ones being combinatorially isomorphic to the octahedron. This leads to the following theorem, see Section~\ref{sec:WithVert}.

\begin{Alphathm}
\label{thm:1}
There exists a compact infinitesimally flexible hyperbolic cone-mani\-fold homeomorphic to the ball, with the singular locus the skeleton of an octahedron.
\end{Alphathm}

A Euclidean polyhedron $P$ whose vertices lie outside the ball of the Klein model can be viewed as a non-compact hyperbolic polyhedron $P^\H$. By truncating the infinite ends, we obtain a compact polyhedron $P^\H_\tr$. If $P$ is infinitesimally flexible, then an analog of Pogorelov's theorem implies that $P^\H_\tr$ is infinitesimally flexible in the class of truncated hyperideal polyhedra. By gluing four copies of $P^\H_\tr$ together, we obtain a cone-manifold without vertices that is again infinitesimally flexible. If the gluing is done in a ``doubling and redoubling'' fashion, then each edge of $P$ gives rise to a component of the singular locus of the manifold. By modifying the gluing scheme, we can reduce the number of components.

\begin{Alphathm}
\label{thm:2}
There exist a compact non-orientable infinitesimally flexible hyperbolic cone-manifold without vertices whose singular locus has three components and a compact orientable infinitesimally flexible hyperbolic cone-manifold without vertices whose singular locus has four components. In both cases, exactly one of the components has cone angle larger than $2\pi$.
\end{Alphathm}

In the light of a result in \cite{HK08} mentioned in Section \ref{subsec:RigThms}, it would be interesting to compute the injectivity radius of the tube around the singular locus in our examples. This amounts to computing the minimum distance between edges of an infinitesimally flexible hyperideal polyhedron. The result of Hodgson and Kerckhoff implies that in every infinitesimally flexible hyperideal polyhedron some pair of edges must have distance at most $\arctanh \frac{1}{\sqrt{3}}$. We don't know if this is bound is sharp.

Once an infinitesimally flexible cone-manifold without vertices is found, other examples can be constructed by taking over it a finite-sheeted branched cover whose branching locus is a subset of the singular locus of the cone-manifold. If the covering space is branched enough over all of the components of the singular locus with cone angles less than $2\pi$, then all cone angles in the covering space are larger than $2\pi$. By applying this idea, we prove in Section \ref{subsec:LargeAngles} the following theorem.

\begin{Alphathm}
\label{thm:3}
There exists a compact infinitesimally flexible hyperbolic cone-mani\-fold without vertices with all cone angles larger than $2\pi$.
\end{Alphathm}

The deaveraging lemma of Pogorelov, shows how an infinitesimally flexible polyhedron gives rise to a pair of non-congruent polyhedra with isometric boundaries. With the help of this, we show that global rigidity fails if cone angles larger than $2\pi$ are allowed.

\begin{Alphathm}
\label{thm:4}
There exist non-isometric compact cone-manifolds $M_1, M_2$ with singular loci $\Sigma_1, \Sigma_2$ such that pairs $(M_1, \Sigma_1)$ and $(M_2, \Sigma_2)$ are homeomorphic and the cone angles at all singular segments in $M_1$ are equal to corresponding cone angles in $M_2$. Besides, $M_1$ and $M_2$ can be chosen arbitrarily close to each other in the Gromov-Hausdorff metric.
\end{Alphathm}

Theorem \ref{thm:4} follows from Proposition \ref{prp:SameAngles}.

\subsection{Acknowledgements}
I would like to thank Jean-Marc Schlenker for constant interest to this work and for suggestion to use branched covers to prove Theorem \ref{thm:3}. I also thank Rafe Mazzeo for an interesting discussion we had on the subject of this paper, and Gr\'{e}goire Montcouquiol for pointing out the work \cite{HK08} and other results.

\section{Infinitesimally flexible cone-manifolds with vertices}
\label{sec:WithVert}
\subsection{Sch\"onhardt's twisted octahedron}
\label{subsec:InfFlexPol}
\begin{dfn}
\label{dfn:InfFlexPol}
Let $P \subset \R^3$ be a polyhedron with triangular faces and vertex set $\cV = \{p_1, \ldots, p_n\}$. An \emph{infinitesimal isometric deformation} of $P$ is a map $q: \cV \to \R^3$ such that
\begin{equation}
\label{eqn:InfFlexPol}
\left.\frac{d}{dt}\right|_{t=0} \dist(p_i+tq_i, p_j+tq_j) = 0,
\end{equation}
for all edges $p_ip_j$ of $P$. Here $q_i$ denotes $q(p_i)$.

An infinitesimal isometric deformation is called \emph{trivial} if $q_i = \xi(p_i)$, for some Killing field (global infinitesimal isometric deformation) $\xi$ of $\R^3$.

A polyhedron $P$ is called \emph{infinitesimally flexible} if it has a non-trivial infinitesimal isometric deformation.
\end{dfn}

A simple calculation shows that condition \eqref{eqn:InfFlexPol} is equivalent to
\begin{equation}
\label{eqn:ScalProd}
\langle p_i - p_j, q_i - q_j \rangle = 0.
\end{equation}

\begin{exl}[Sch\"onhardt \cite{Scn28}, Wunderlich \cite{Wun65}]
\label{exl:TwistOcta}
Let $ABC$ be an equilateral triangle in $\R^3$, and let $l$ be a line that passes through the center of $ABC$ orthogonally to the plane of the triangle. Let $A'B'C'$ be the image of $ABC$ under a screw motion with axis $l$ and rotation angle $\frac{\pi}{2}$. Consider a polyhedron $P$ bounded by triangles $ABC$, $A'B'C'$, $ABC'$, $A'BC$, $AB'C$, $A'B'C$, $AB'C'$, and $A'BC'$. The polyhedron $P$ is combinatorially isomorphic to an octahedron, and has three edges with dihedral angles bigger than $\pi$: the edges $AB'$, $BC'$, and $CA'$, see Figure \ref{fig:TwistOcta}.

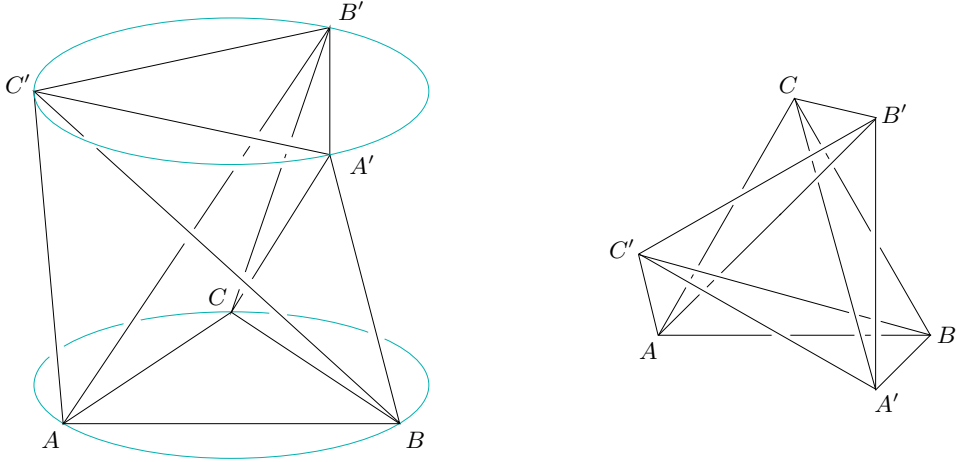
\begin{figure}[ht]
\begin{center}
\input{Fig/TwistOcta.tex}
\end{center}
\caption{Sch\"onhardt's twisted octahedron.}
\label{fig:TwistOcta}
\end{figure}

\end{exl}

\begin{lem}
\label{lem:InfFlex}
The polyhedron from Example \ref{exl:TwistOcta} is infinitesimally flexible.
\end{lem}
\begin{proof}
Put $q(A) = q(B) = q(C) = 0$. Let $q(C')$ be a vector orthogonal to the plane $ABC'$, and let $q(A')$ and $q(B')$ be images of $q(C')$ under rotations by $\frac{2\pi}{3}$ and $\frac{4\pi}{3}$ around $l$. Clearly, the infinitesimal deformation $q$ preserves in the first order the side lengths of the triangles $ABC'$, $A'BC$, and $AB'C$. Let us show that the side lengths of the triangle $A'B'C'$ are also infinitesimally preserved. Indeed, the plane $ABC'$ is easily seen to pass through the center of the triangle $A'B'C'$. Therefore the vector $q(C')$ is tangent to a cylinder with axis $l$. By symmetry, the triangle $A'B'C'$ undergoes an infinitesimal screw motion with axis $l$. Hence, its side lengths are also infinitesimally constant.
\end{proof}

\subsection{Infinitesimal Pogorelov map}
Infinitesimally flexible hyperbolic po\-lyhedra are defined similarly to Euclidean ones, see Definition \ref{dfn:InfFlexPol}. We assume $q_i \in T_{p_i}\H^3$ and replace $p + tq$ in \eqref{eqn:InfFlexPol} by $\exp_p(tq)$, where
$$
\exp_p \co T_p\H^3 \to \H^3
$$
is the exponential map. Pogorelov \cite[Chapter 5]{Pog73} showed that a hyperbolic polyhedron is infinitesimally flexible if and only if its image in a Klein model is an infinitesimally flexible Euclidean polyhedron.

\begin{thm}[Pogorelov]
\label{thm:Pogorelov}
Let $P \subset \R^3$ be a Euclidean polyhedron. Take an arbitrary ball $B \subset \R^3$ that contains $P$ in its interior, and regard $B$ as a Klein model of the hyperbolic space $\H^3$. Denote by $P^\H \subset \H^3$ the hyperbolic polyhedron that corresponds to $P \subset B$. Then $P^\H$ is infinitesimally flexible if and only if $P$ is infinitesimally flexible.

More precisely, there is a canonical way to associate with an infinitesimal isometric deformation $q$ of $P$ an infinitesimal isometric deformation $q^\H$ of $P^\H$, so that $q^\H$ is trivial if and only if $q$ is trivial.
\end{thm}

Pogorelov considered infinitesimal isometric deformations of smooth surfaces, but his arguments carry over to the polyhedral case.

The correspondence $q \mapsto q^H$ between infinitesimal isometric deformations of $P$ and $P^\H$ is called \emph{infinitesimal Pogorelov map}. Its existence can be explained by the fact that infinitesimal flexibility is a projective property, rather than a metric one, see \cite{Izm09}.

By regarding the polyhedron from Example \ref{exl:TwistOcta} as the image of a hyperbolic polyhedron in a Klein model, we obtain the following statement.

\begin{cor}
\label{cor:FlexHypOcta}
There exists an infinitesimally flexible hyperbolic polyhedron combinatorially equivalent to the octahedron.
\end{cor}

\begin{rem}
Note that the infinitesimal Pogorelov map $q \mapsto q^\H$ is not the identity in the Klein model $B$. For example, the vectors $q(A')$, $q(B')$, and $q(C')$ in the proof of Lemma \ref{lem:InfFlex} are tangent to a Euclidean cylinder, whereas the corresponding vectors $q^\H(A')$, $q^\H(B')$, and $q^\H(C')$ will be tangent to an equidistant surface of the line $l$ in the hyperbolic metric of the Klein model (assuming that the center of $P$ is the center of the Klein model).
\end{rem}

\subsection{Flexibility of the double}
\label{subsec:FlexDouble}
Let $P$ be a hyperbolic polyhedron. Let $M$ be the \emph{double} of $P$, that is the result of gluing $P$ and its isometric copy $P'$ along pairs of corresponding faces. Then $M$ is a hyperbolic cone-manifold homeomorphic to a $3$--dimensional sphere, and its singular locus is the skeleton of $P$.

\begin{prp}
\label{prp:FlexDouble}
The double of an infinitesimally flexible polyhedron is an infinitesimally flexible cone-manifold with vertices.
\end{prp}

A non-trivial infinitesimal deformation of $M$ can be described as follows. Let $q$ be a non-trivial infinitesimal isometric deformation of $P$. The deformation $q$ preserves the edge lengths of $P$ in the first order, but changes the dihedral angles (otherwise it can be shown that $q$ is trivial). The opposite deformation $-q$, which consists of vectors $-q_i$, also preserves the edge lengths. A crucial point is that the variations of dihedral angles under $-q$ are the negatives of their variations under $q$. Thus if $M = P \cup P'$ and we deform $P$ and $P'$ according to $q$ and $-q$ respectively, then the angles around the segments of the singular locus are preserved in the first order. The resulting deformation is non-trivial, because the links of the vertices are deformed non-trivially.

Note that this infinitesimal deformation of $M$ preserves not only the angles around the edges, but also their lengths.

Corollary \ref{cor:FlexHypOcta} and Proposition \ref{prp:FlexDouble} imply Theorem \ref{thm:1}.

\subsection{Other examples of infinitesimally flexible polyhedra}

There is an elegant description of all infinitesimally flexible octahedra.

\begin{thm}[Blaschke \cite{Bla20}, Liebmann \cite{Lie20}]
\label{thm:BlaLie}
Let $P \subset \R^3$ be a polyhedron combinatorially isomorphic to the octahedron. Color the faces of $P$ black and white so that every pair of adjacent faces has different colors. Then $P$ is infinitesimally flexible if and only if the four black faces intersect at a point or, equivalently, if the four white faces intersect at a point. Intersection points can lie at infinity.
\end{thm}

For example, on Figure \ref{fig:TwistOcta} the faces $ABC'$, $A'BC$, $AB'C$, and $A'B'C'$ intersect at the center of the face $A'B'C'$. Figure \ref{fig:Gluck} shows another example, taken from \cite{Glu75}. It is easy to see that the intersection point of the lines $AB$ and $CD$ is common to all four black faces.

\begin{figure}[ht]
\begin{center}
\input{Fig/Gluck.tex}
\end{center}
\caption{If the points $A$, $B$, $C$, and $D$ lie in one plane, then this polyhedron is infinitesimally flexible.}
\label{fig:Gluck}
\end{figure}
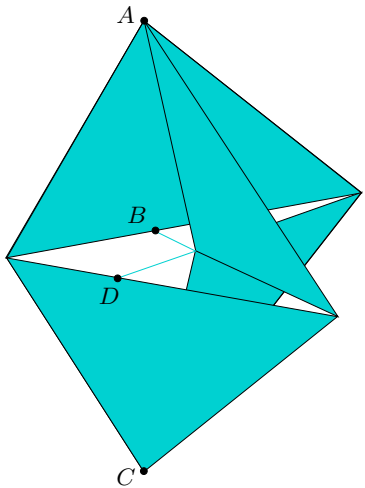

\begin{cor}
A hyperbolic polyhedron combinatorially isomorphic to the octahedron is infinitesimally flexible if and only if its four black faces (see Theorem \ref{thm:BlaLie}) intersect at a point. The intersection point can lie at infinity or beyond infinity.
\end{cor}
\begin{proof}
This is a direct consequence of Theorems \ref{thm:Pogorelov} and \ref{thm:BlaLie}.
\end{proof}

A generalization of Example \ref{exl:TwistOcta} in a different direction is a \emph{twisted antiprism}. A regular antiprism is the convex hull of a regular $n$-gon $C$ and of its image $C'$ under a screw motion with rotation angle $\frac{\pi}{n}$ and axis orthogonal to $C$ and passing through its center. A twisted antiprism is obtained from the regular one by rotating the base $C'$ by $\frac{\pi}{2}$ around its center while preserving the combinatorics. A twisted antiprism has a non-trivial infinitesimal isometric deformation similar to that described in the proof of Lemma \ref{lem:InfFlex}.

One more example of an infinitesimally flexible polyhedron is Jessen's \emph{orthogonal icosahedron}, \cite{Gol78}.

\section{Infinitesimally flexible cone-manifolds without vertices}
\label{sec:InfFlexWOVert}
\subsection{Infinitesimal deformations of truncated hyperideal polyhedra}
\label{subsec:InfDefHyper}
Let $B \in \R^3$ be a ball, and $P \subset \R^3$ be a polyhedron such that all vertices of $P$ lie outside $B$, and all edges of $P$ intersect the interior of $B$. Regard $B$ as a Klein model of $\H^3$. Then the intersection of $P$ with the interior of $B$ is the image of a \emph{hyperideal} polyhedron $P^\H \subset \H^3$. For every vertex $p_i$ of $P$, denote by $p_i^*$ a plane in $\R^3$ polar to $p_i$ with respect to the boundary sphere of $B$. The corresponding hyperbolic plane, that we also denote $p_i^*$, intersects the boundary of the polyhedron $P^\H$ orthogonally. The plane $p_i^*$ cuts off an infinite end of $P^\H$ incident to the (hyperideal) point $p_i$. By cutting off all infinite ends we obtain a \emph{truncated hyperideal polyhedron} which we denote by $P^\H_\tr$.

Let $p_ip_j$ be an edge of the polyhedron $P$. Since $p_ip_j$ intersects the interior of $B$, the hyperbolic planes $p_i^*$ and $p_j^*$ don't intersect each other. It follows that the truncated hyperideal polyhedron $P^\H_\tr$ is combinatorially isomorphic to $P$ with small neighborhoods of vertices removed. See Figure \ref{fig:TruncOcta} for a combinatorial structure of a truncated hyperideal octahedron.

Denote by $p_{ij}$ the vertex of $P^\H_\tr$ that is the intersection point of the plane $p_i^*$ with the edge $p_ip_j$ of $P$.
Let us refer to the faces of $P^\H_\tr$ that are subsets of faces of $P$ as \emph{old} faces, and to the faces that span the planes $p_i^*$ as \emph{new} ones. Similarly, if an edge is a part of an edge of $P$, call it old; call all the other edges, that is all edges of new faces, new. Note that the dihedral angles at new edges are all equal to $\frac{\pi}{2}$.

Assume that all faces of the polyhedron $P$ are triangular. We are going to define infinitesimal deformations of $P^\H_\tr$ \emph{in the class of truncated hyperideal polyhedra}.

\begin{dfn}
\label{dfn:DefTruncPol}
Let $P^\H_\tr$ be a truncated hyperideal polyhedron such that all faces of $P$ are triangles. An \emph{infinitesimal deformation} of $P^\H_\tr$ in the class of truncated hyperideal polyhedra is an assignment $p_{ij} \mapsto q_{ij} \in T_{p_{ij}}\H^3$ such that if all of the vertices move along trajectories $p_{ij}(t) = \exp_{p_{ij}}(tq_{ij})$, then
\begin{enumerate}
\item \label{item1} all faces of $P^\H_\tr$ remain planar;
\item \label{item2} all angles between old and new faces remain $\frac{\pi}{2}$.
\end{enumerate}

An infinitesimal deformation is called \emph{isometric}, if in addition
\begin{enumerate}
\setcounter{enumi}{2}
\item \label{item3} lengths of all old edges are constant in the first order.
\end{enumerate}

A truncated hyperideal polyhedron $P^\H_\tr$ is called \emph{infinitesimally flexible} if it posesses an infinitesimal isometric deformation (in the class of truncated hyperideal polyhedra) that is not a restriction of a Killing field.
\end{dfn}

Some remarks are in order. First, it seems more consistent to require in conditions (\ref{item1}) and (\ref{item2}) that faces remain flat and angles are preserved in the first order only. However, it is easy to see that if several of the vertices $p_{ij}(t)$ remain infinitesimally coplanar, then they just remain coplanar. Similarly, if a dihedral angle has zero derivative under this deformation, then this angle is constant.

Second, condition (\ref{item3}) implies that lengths of new edges are also constant in the first order. Indeed, each old face of $P^\H_\tr$ is a right-angled hexagon, and lengths of three pairwise disjoint of its edges determine the lengths of the other three edges. Moreover, the Jacobian of the map that associates three new lengths to three old lengths is non-degenerate. This implies our assertion.

And third, if an isometric infinitesimal deformation is non-trivial, then some angles of new faces must have a non-zero variation. That is, an isometric infinitesimal deformation in the class of truncated hyperideal polyhedra is isometric only on the old part of the boundary.

\subsection{De Sitter space and infinitesimal Pogorelov map}

Let
$$
\langle x, y \rangle_{(3,1)} = -x_0y_0 + x_1y_1 + x_2y_2 + x_3y_3
$$
be the scalar product in the Minkowski space $\R^{3,1}$. We identify $\H^3$ with its hyperboloid model:
$$
\H^3 = \{x \in \R^{3,1}\ |\ \langle x, x \rangle_{(3,1)} = -1, x_0 > 0\},
$$
with the induced metric. The \emph{de Sitter space} is the one-sheeted hyperboloid with the induced metric:
$$
\dS^3 = \{x \in \R^{3,1}\ |\ \langle x, x \rangle_{(3,1)} = 1\}.
$$
The central projection from the origin maps $\H^3$ to the interior of the unit disk on the hyperplane $\{x_0 = 1\}$, and this yields the Klein model. The same projection maps the upper half of the de Sitter space to the exterior of the unit disk, which we thus consider as the \emph{Klein model of the de Sitter space}.

Let $P$ be a polyhedron as described at the beginning of Section \ref{subsec:InfDefHyper}. Then vertices of $P$ represent points in the de Sitter space. Polyhedron $P$ itself is the image of a \emph{hyperbolic-de Sitter polyhedron} $P^{\H\dS}$. We define an infinitesimal deformation of $P^{\H\dS}$ as a collection of vectors $q_i \in T_{p_i}\dS^3$. Every infinitesimal deformation of $P^{\H\dS}$ induces an infinitesimal deformation of $P^\H_\tr$ in the class of truncated hyperideal polyhedra.

\begin{lem}
\label{lem:DefDeSit}
Every infinitesimal deformation of $P^\H_\tr$ in the class of truncated hyperideal polyhedra is induced by an infinitesimal deformation of $P^{\H\dS}$.

An infinitesimal deformation of $P^\H_\tr$ is isometric if and only if for the corresponding deformation of $P^{\H\dS}$ we have
\begin{equation}
\label{eqn:DeSitIsom}
\langle p_i - p_j, q_i - q_j \rangle_{(3,1)} = 0,
\end{equation}
for all edges $ij$ of $P$.
\end{lem}
\begin{proof}
Consider a new face of $P^\H_\tr$. By condition (\ref{item1}) in Definition \ref{dfn:DefTruncPol}, the face remains planar during the deformation. The point dual to the plane of the face traces a trajectory $p_i(t)$ in the de Sitter space. Since, by condition (\ref{item2}), all incident old faces remain orthogonal to the new face, their planes pass through the point $p_i(t)$. Thus the vectors $q_i = \left.\frac{d}{dt}\right|_{t=0} p_i(t)$ induce the given infinitesimal deformation of $P^\H_\tr$.

Since the old edges of $P^\H_\tr$ are orthogonal to the new faces they join, we have for their lengths
$$
\dist(p_{ij}, p_{ji}) = \dist(p_i^*, p_j^*).
$$
On the other hand, it is well known that
$$
-\cosh \dist(p_i^*, p_j^*) = \langle p_i, p_j \rangle_{(3,1)}.
$$
Therefore condition (\ref{item2}):
$$
\left.\frac{d}{dt}\right|_{t=0} \dist(p_{ij}(t), p_{ji}(t)) = 0
$$
is equivalent to
$$
\left.\frac{d}{dt}\right|_{t=0} \langle p_i(t), p_j(t) \rangle_{(3,1)} = 0
$$
which due to $\langle p_i, q_i \rangle_{(3,1)} = \langle p_j, q_j \rangle_{(3,1)} = 0$ can be rewritten as
$$
\langle p_i - p_j, q_i - q_j \rangle_{(3,1)} = 0.
$$
\end{proof}

\begin{rem}
The points $p_{ij}(t)$, by Definition \ref{dfn:DefTruncPol}, move with constant velocities along geodesics. The corresponding points $p_i(t)$ don't have constant velocities in general.
\end{rem}

It is natural to call an infinitesimal deformation $\{q_i \in T_{p_i}\dS^3\}$ of $P^{\H\dS}$ \emph{isometric} if the condition \eqref{eqn:DeSitIsom} holds. As usual, we say that a deformation is trivial if it is a restriction of a Killing field on $\dS^3$. Thus we come to a notion of infinitesimal rigidity of hyperbolic-de Sitter polyhedra.

\begin{prp}
\label{prp:DeSitPog}
A hyperbolic-de Sitter polyhedron $P^{\H\dS}$ is infinitesimally flexible if and only if the corresponding Euclidean polyhedron $P$ is infinitesimally flexible.
\end{prp}

The theorem can be proved either by repeating the arguments from \cite[Section 4.1]{Izm09} or by exhibiting a formula that relates infinitesimal isometric deformations of $P^{\H\dS}$ to those of $P$, \cite[Section 1]{Sch05}.

\begin{cor}
\label{cor:FlexTruncOcta}
There exists a truncated hyperideal octahedron infinitesimally flexible in the sense of Definition \ref{dfn:DefTruncPol}.
\end{cor}
\begin{proof}
Sch\"onhardt's twisted octahedron (Example \ref{exl:TwistOcta}) can be inscribed in a ball. By taking a ball of a slightly smaller radius and with the same center, we obtain a hyperideal polyhedron. Its truncation is infinitesimally flexible due to Lemma \ref{lem:DefDeSit} and Proposition \ref{prp:DeSitPog}.
\end{proof}

\subsection{Doubling the double}
\label{subsec:DD}
Let $P^\H_\tr$ be a truncated hyperideal polyhedron. Take its double $D$ along old faces. Then $D$ is a cone-manifold with geodesic boundary: old edges of $P^\H_\tr$ give rise to singular segments with endpoints on the boundary, and the boundary is geodesic because all angles between old and new faces are $\frac{\pi}{2}$. It follows that the double $M$ of $D$ is a cone-manifold without vertices. Components of the singular locus of $M$ are in one-to-one correspondence with old edges of $P^\H_\tr$, with cone angles twice the dihedral angles.

\begin{prp}
If the polyhedron $P^\H_\tr$ is infinitesimally flexible in the class of truncated hyperideal polyhedra, then $M$ is an infinitesimally flexible cone-manifold.
\end{prp}
\begin{proof}
Deform one half of $D$ according to a non-trivial infinitesimal deformation $q$ of $P^\H_\tr$, and the other half according to the opposite deformation $-q$. This gives an infinitesimal deformation of $D$ that preserves the angles around the segments of the singular locus, in the first order. The boundary of $D$ undergoes a non-isometric infinitesimal deformation but remains geodesic. Thus if we take the same infinitesimal deformation on an isometric copy of $D$, then together they yield an infinitesimal deformation of $M$ that preserves the angles around the components of the singular locus. Note that the lengths of the components of the singular locus are also preserved.
\end{proof}

If we take for $P^\H_\tr$ an infinitesimally flexible truncated octahedron that exists by Corollary \ref{cor:FlexTruncOcta}, then the singular locus of the manifold $M$ has 12 components. For a symmetric truncated octahedron described in the proof of Corollary \ref{cor:FlexTruncOcta}, the gluing pattern can be modified so that to obtain an infinitesimally flexible manifold with less components of the singular locus.

A symmetric infinitesimally flexible truncated octahedron has three types of dihedral angles, see Figure \ref{fig:TruncOcta}. Letters on the new faces correspond to the hyperideal vertices, see Figure \ref{fig:TwistOcta}. Note that all new faces are equal, equally oriented, and undergo the same infinitesimal deformation when $P^\H_\tr$ is isometrically infinitesimally deformed.

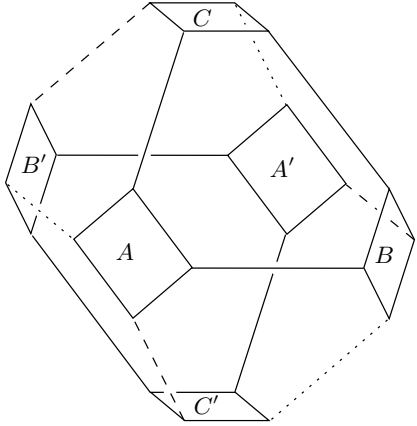
\begin{figure}[ht]
\begin{center}
\input{Fig/TruncOcta.tex}
\end{center}
\caption{A schematic drawing of truncated twisted octahedron. Edges with equal dihedral angles are drawn with lines of same type.}
\label{fig:TruncOcta}
\end{figure}

\begin{exl}
\label{exl:3comp}
Instead of doubling $D$, identify its boundary components pairwise. Namely, the double of the face $A$ is glued to the double of the face $A'$ and so on, and we take care that a face of $P^\H_\tr$ is each time glued to a face of $P^\H_\tr$ but not to a face of its double. Then the gluing is consistent with an infinitesimal deformation of $D$ that preserves the angles around the singular locus.

The same gluing pattern can be described as first identifying antipodal pairs of new faces of $P^\H_\tr$ and then doubling the result. After the first step we obtain a non-orientable infinitesimally flexible manifold with polyhedral boundary. The boundary has no vertices and exactly three edges as can be seen on Figure \ref{fig:TruncOcta} by tracing the identifications of edges' endpoints.
\end{exl}

\begin{exl}
\label{exl:4comp}
The two-fold orientable cover of the previous example can be described as follows. We take two copies of $D$ and glue them along the antipodal pairs of boundary components. For each antipodal pair, there are two possible gluings, and we choose them consistently, so that each copy of $P^\H_\tr$ is glued along its new faces to only one of the other copies. This ensures that infinitesimal isometric deformations $\pm q$ on all four copies can be chosen so that they fit together and leave the cone angles unchanged in the first order.

By tracing the gluing with the help of Figure \ref{fig:TruncOcta}, it can be shown that the singular locus consists of four components, and only one of them has cone angle bigger than $2\pi$.
\end{exl}

Examples \ref{exl:3comp} and \ref{exl:4comp} prove Theorem \ref{thm:2}.

\subsection{An example with all cone angles larger than $2\pi$}
\label{subsec:LargeAngles}
At the beginning of Section \ref{subsec:DD} we described a cone-manifold $M$ glued from four copies of a truncated hyperideal polyhedron $P^\H_\tr$. First we double $P^\H_\tr$ along old faces to obtain a cone-manifold with boundary $D$, then we double $D$. To obtain a cone-manifold with all angles larger than $2\pi$, we will construct a cover of $D$ branched over the singular segments with cone angles less than $2\pi$.

\begin{exl}
\label{exl:NegCurv}
By Lemma \ref{lem:AngleEst}, there exists an infinitesimally flexible truncated hyperideal octahedron $P^\H_\tr$ with all angles larger than $\frac{\pi}{7}$. Let $D$ be the double of $P^\H_\tr$ along old faces, and let $\Sigma_+ \subset D$ be the union of singular segments with cone angles less than $2\pi$. The space $D \setminus \Sigma_+$ is homeomorphic to the complement of the skeleton of a triangular prism in $\Sph^3$. By Lemma \ref{lem:7Fold}, there exists a seven-fold cover
$$
\widehat{D \setminus \Sigma_+} \to D \setminus \Sigma_+
$$
such that links of all components of $\Sigma_+$ are covered non-trivially. Then the completion $D' = \widehat{D \setminus \Sigma_+} \cup \Sigma_+$ is a cone-manifold with boundary and with cone angles around the components of $\Sigma_+$ seven times larger than those in $D$. Since all cone angles in $D$ are larger than $\frac{2\pi}{7}$, all cone angles in $D'$ are larger than $2\pi$.

Let $M'$ be the double of $D'$. Then $M'$ is a cone-manifold without boundary and with all cone angles larger than $2\pi$. Since $M'$ is a cover of $M$ branched over a subset of the singular locus, infinitesimal flexibility of $M$ implies infinitesimal flexibility of $M'$.
\end{exl}

\begin{lem}
\label{lem:AngleEst}
For every $\epsilon > 0$, there exists an infinitesimally flexible truncated hyperideal octahedron with all dihedral angles larger than $\frac{\pi}{6} - \epsilon$.
\end{lem}
\begin{proof}
It suffices to prove this for ideal octahedra, because slightly pushing the vertices outwards changes the dihedral angles only a little. Consider an ideal octahedron $P_\id$ obtained from Example \ref{exl:TwistOcta}. In the Poincar\'{e} half-space model, vertices of $P_\id$ are the vertices of two concentric regular triangles, with pairwise orthogonal sides, see Figure \ref{fig:IdOcta}, left. A face of $P_\id$ intersects the boundary plane of the model in a circle, and the dihedral angle between two faces is equal to the angle between corresponding circles. On Figure \ref{fig:IdOcta}, the four angles adjacent to the vertex $A$ are marked.

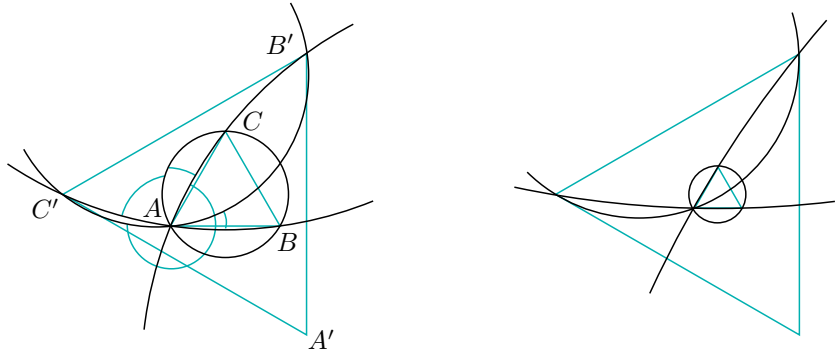
\begin{figure}[ht]
\begin{center}
\input{Fig/IdOcta.tex}
\end{center}
\caption{Dihedral angles of an ideal twisted octahedron.}
\label{fig:IdOcta}
\end{figure}

The ideal polyhedron $P_\id$ depends on one parameter, the relative size of the triangles. As this parameter tends to infinity, the dihedral angles of $P_\id$ tend to $\frac{\pi}{6}$, $\frac{\pi}{3}$, $\frac{\pi}{3}$, and $\frac{7\pi}{6}$, see Figure \ref{fig:IdOcta}, right. Thus, for a sufficiently large value of the parameter, all angles are larger than $\frac{\pi}{6} - \epsilon$. Lemma follows.
\end{proof}

\begin{lem}
\label{lem:7Fold}
Let $\Gamma \subset \Sph^3$ be the skeleton of a triangular prism. There exists a seven-fold cover of $\Sph^3 \setminus \Gamma$ such that the link of every edge of $\Gamma$ is covered non-trivially (that is, the preimage of a meridional curve of every edge has one component).
\end{lem}
\begin{proof}
It suffices to construct a homomorphism
$$
\phi \co \pi_1(\Sph^3 \setminus \Gamma) \to \Z_7
$$
that maps every meridional element of $\pi_1(\Sph^2 \setminus \Gamma)$ to a non-zero element of~$\Z_7$. The fundamental group of $\Sph^3 \setminus \Gamma$ is freely generated by the meridians $a_1, \ldots, a_4$ of a four-cycle in $\Gamma$, all the other meridians being words of length two, see Figure \ref{fig:FundGroup}. Let $\phi$ be a homomorphism defined by sending $a_1$, $a_2$, and $a_4$ to $1 \in \Z_7$, and $a_3$ to $2 \in \Z_7$. Then the images of all meridians are contained in the set $\{1,2,3\}$. Lemma is proved.
\end{proof}

\begin{figure}[ht]
\begin{center}
\input{Fig/FundGroup.tex}
\end{center}
\caption{Meridional elements of $\pi_1(\Sph^3 \setminus \Gamma)$.}
\label{fig:FundGroup}
\end{figure}
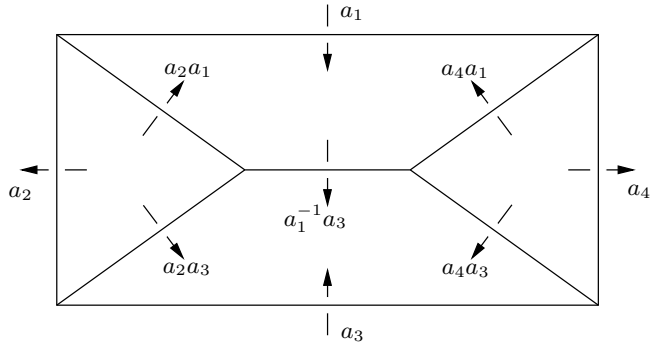

\section{Pairs of non-isometric cone-manifolds with same cone angles}
\subsection{Deaveraging}
\label{subsec:Deave}
\begin{lem}[Pogorelov]
\label{lem:Deave}
Let $P$ be a Euclidean polyhedron and let $q$ be an infinitesimal isometric deformation of $P$, see Definition \ref{dfn:InfFlexPol}. Denote by $P_t$ a polyhedron with the same combinatorics as $P$ and with vertices $p_i(t) = p_i + tq_i$ instead of $p_i$. Then, for all $t>0$, the pairs of corresponding edges of $P_t$ and $P_{-t}$ have equal lengths. Besides, $q$ is non-trivial if and only if polyhedra $P_t$ and $P_{-t}$ are non-congruent.
\end{lem}

To play safe, Lemma \ref{lem:Deave} should be stated for all sufficiently small $t$, as for a big $t$ the polyhedron $P_t$ might be self-intersecting.

The equalities
$$
\dist(p_i(t), p_j(t)) = \dist(p_i(-t), p_j(-t))
$$
for all edges $ij$ of $P$ follow easily from \eqref{eqn:ScalProd}. Pogorelov uses in \cite[Chapter 5]{Pog73} the inverse of Lemma \ref{lem:Deave} in the smooth case, but the idea transfers to the discrete situation easily.

Finally, Lemma \ref{lem:Deave} has a hyperbolic and a hyperbolic-de Sitter analogs, with
$$
p_i(t) = \exp_{p_i}(tq_i).
$$

\subsection{Existence of non-isometric pairs}
\label{subsec:NonIsom}
Using the hyperbolic version of Lemma \ref{lem:Deave}, we now construct three families of cone-manifolds. Let $P$ be a compact hyperbolic polyhedron with a non-trivial infinitesimal isometric deformation $q$, and let $P_t$ and $P_{-t}$ be polyhedra constructed in Lemma \ref{lem:Deave}. Denote
\begin{eqnarray}
P_t \cup_{\partial} P_t & = & M^1_t,\nonumber\\
P_t \cup_{\partial} P_{-t} & = & M^2_t,\label{eqn:ThreeFam}\\
P_{-t} \cup_{\partial} P_{-t} & = & M^3_t.\nonumber
\end{eqnarray}
Here $\cup_\partial$ means gluing of two polyhedra along their boundaries. Denote by $M^i_0$ the double $P \cup_\partial P$. The manifold $M^i_0$ does not depend on the choice of the index $i$ and can be viewed as a common element of the three families~\eqref{eqn:ThreeFam}.

Specifically, consider a symmetric twisted octahedron obtained from a Euclidean polyhedron in Example \ref{exl:TwistOcta} by placing it in the center of a Klein model. By taking Euclidean polyhedra of different widths and heights, we obtain a two-parameter family of hyperbolic twisted octahedra $P(a,b)$. Here $a$ stands e.g. for the edge length of the base, and $b$ for the distance between the bases. As in the previous paragraph, for each sufficiently small $t > 0$ we have cone-manifolds $M^i_t(a,b)$, $i = 1,2,3$.

\begin{prp}
\label{prp:SameAngles}
For every $\epsilon > 0$ and every pair of intervals $A, B \subset \R$ there exist $t_1, t_2 \in (0,\epsilon]$, $a_1, a_2 \in A$, and $b_1, b_2 \in B$ such that cone-manifolds $M_{t_1}^{i}(a_1,b_1)$ and $M_{t_2}^{j}(a_2,b_2)$, for some $i \ne j$, have same cone angles.
\end{prp}
\begin{proof}
Consider a space of cone-manifolds
$$
\cM = \{M_t^i(a,b)\ |\ t \in [0, \epsilon], a \in A, b \in B, i \in \{1,2,3\}\}.
$$
We have
\begin{equation}
\label{eqn:Triod}
\cM \approx A \times B \times Y \approx \Disk^2 \times Y,
\end{equation}
where $Y$ is a ``triod'', the topological space homeomorphic to the union of three copies of $[0,1]$ with all exemplars of $0$ identified.

Due to the symmetry of $P(a,b)$, the singular segments of each of the cone-manifolds $M_t^i(a,b)$ split into three groups according to the values of cone angles around them. Thus we have a continuous map
$$
\cM \to \R^3,
$$
and we have to show that this map is not injective. But it cannot be injective due to \eqref{eqn:Triod}. Proposition is proved.
\end{proof}

By deaveraging an infinitesimally flexible hyperideal polyhedron, one can prove existence of non-isometric pairs of cone-manifolds without vertices and with same cone angles. By taking a suitable branched cover as in Section \ref{subsec:LargeAngles}, all cone angles can be made larger than $2\pi$.


\end{document}

%% file: Fig/TwistOcta.tex
\begin{picture}(0,0)%
\includegraphics{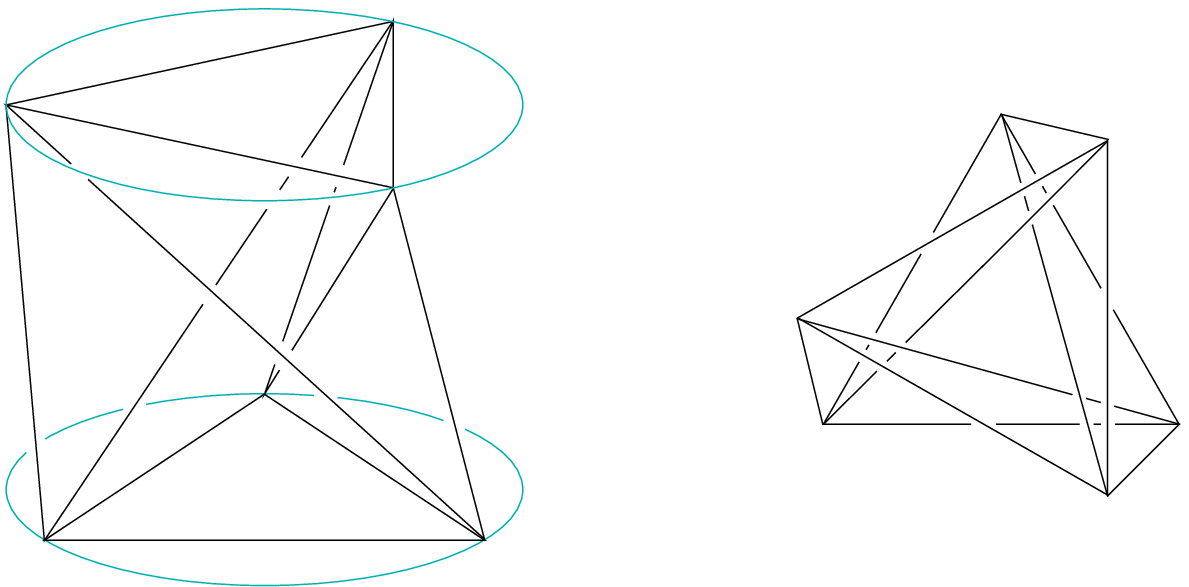}%
\end{picture}%
\setlength{\unitlength}{3108sp}%
\begingroup\makeatletter\ifx\SetFigFont\undefined%
\gdef\SetFigFont#1#2#3#4#5{%
  \reset@font\fontsize{#1}{#2pt}%
  \fontfamily{#3}\fontseries{#4}\fontshape{#5}%
  \selectfont}%
\fi\endgroup%
\begin{picture}(7252,3652)(645,-3284)
\put(3802,-3192){\makebox(0,0)[rb]{\smash{{\SetFigFont{9}{10.8}{\rmdefault}{\mddefault}{\updefault}{\color[rgb]{0,0,0}$B$}%
}}}}
\put(2223,-2052){\rotatebox{360.0}{\makebox(0,0)[rb]{\smash{{\SetFigFont{9}{10.8}{\rmdefault}{\mddefault}{\updefault}{\color[rgb]{0,0,0}$C$}%
}}}}}
\put(3322,221){\makebox(0,0)[rb]{\smash{{\SetFigFont{9}{10.8}{\rmdefault}{\mddefault}{\updefault}{\color[rgb]{0,0,0}$B'$}%
}}}}
\put(661,-359){\rotatebox{360.0}{\makebox(0,0)[rb]{\smash{{\SetFigFont{9}{10.8}{\rmdefault}{\mddefault}{\updefault}{\color[rgb]{0,0,0}$C'$}%
}}}}}
\put(5508,-2482){\makebox(0,0)[lb]{\smash{{\SetFigFont{9}{10.8}{\rmdefault}{\mddefault}{\updefault}{\color[rgb]{0,0,0}$A$}%
}}}}
\put(3403,-1006){\makebox(0,0)[rb]{\smash{{\SetFigFont{9}{10.8}{\rmdefault}{\mddefault}{\updefault}{\color[rgb]{0,0,0}$A'$}%
}}}}
\put(888,-3191){\makebox(0,0)[rb]{\smash{{\SetFigFont{9}{10.8}{\rmdefault}{\mddefault}{\updefault}{\color[rgb]{0,0,0}$A$}%
}}}}
\put(5479,-1676){\rotatebox{360.0}{\makebox(0,0)[rb]{\smash{{\SetFigFont{9}{10.8}{\rmdefault}{\mddefault}{\updefault}{\color[rgb]{0,0,0}$C'$}%
}}}}}
\put(7587,-2903){\makebox(0,0)[rb]{\smash{{\SetFigFont{9}{10.8}{\rmdefault}{\mddefault}{\updefault}{\color[rgb]{0,0,0}$A'$}%
}}}}
\put(7882,-2356){\makebox(0,0)[lb]{\smash{{\SetFigFont{9}{10.8}{\rmdefault}{\mddefault}{\updefault}{\color[rgb]{0,0,0}$B$}%
}}}}
\put(7653,-599){\makebox(0,0)[rb]{\smash{{\SetFigFont{9}{10.8}{\rmdefault}{\mddefault}{\updefault}{\color[rgb]{0,0,0}$B'$}%
}}}}
\put(6622,-361){\rotatebox{360.0}{\makebox(0,0)[lb]{\smash{{\SetFigFont{9}{10.8}{\rmdefault}{\mddefault}{\updefault}{\color[rgb]{0,0,0}$C$}%
}}}}}
\end{picture}%

%% file: Fig/Gluck.tex
\begin{picture}(0,0)%
\includegraphics{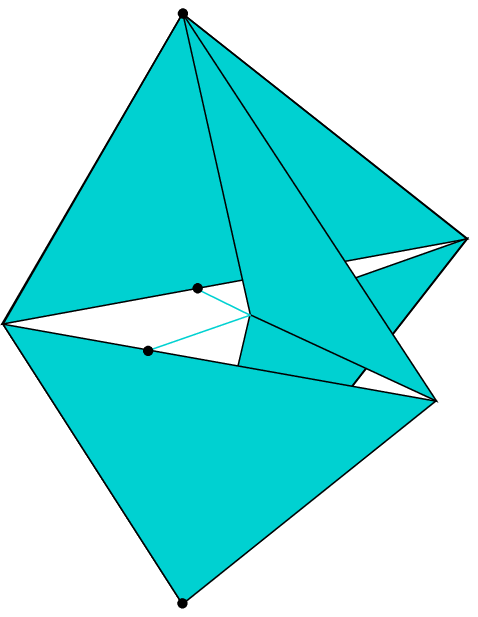}%
\end{picture}%
\setlength{\unitlength}{3108sp}%
\begingroup\makeatletter\ifx\SetFigFont\undefined%
\gdef\SetFigFont#1#2#3#4#5{%
  \reset@font\fontsize{#1}{#2pt}%
  \fontfamily{#3}\fontseries{#4}\fontshape{#5}%
  \selectfont}%
\fi\endgroup%
\begin{picture}(2864,3901)(5629,-3503)
\put(6606,-1343){\makebox(0,0)[lb]{\smash{{\SetFigFont{9}{10.8}{\rmdefault}{\mddefault}{\updefault}{\color[rgb]{0,0,0}$B$}%
}}}}
\put(6522,-3439){\makebox(0,0)[lb]{\smash{{\SetFigFont{9}{10.8}{\rmdefault}{\mddefault}{\updefault}{\color[rgb]{0,0,0}$C$}%
}}}}
\put(6522,251){\makebox(0,0)[lb]{\smash{{\SetFigFont{9}{10.8}{\rmdefault}{\mddefault}{\updefault}{\color[rgb]{0,0,0}$A$}%
}}}}
\put(6382,-1992){\makebox(0,0)[lb]{\smash{{\SetFigFont{9}{10.8}{\rmdefault}{\mddefault}{\updefault}{\color[rgb]{0,0,0}$D$}%
}}}}
\end{picture}%

%% file: Fig/TruncOcta.tex
\begin{picture}(0,0)%
\includegraphics{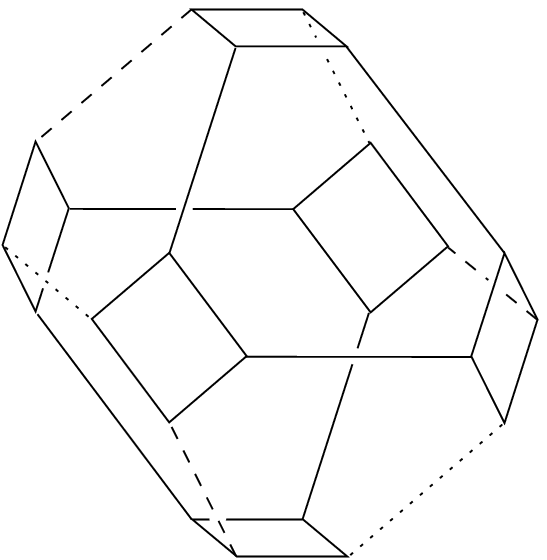}%
\end{picture}%
\setlength{\unitlength}{4144sp}%
\begingroup\makeatletter\ifx\SetFigFont\undefined%
\gdef\SetFigFont#1#2#3#4#5{%
  \reset@font\fontsize{#1}{#2pt}%
  \fontfamily{#3}\fontseries{#4}\fontshape{#5}%
  \selectfont}%
\fi\endgroup%
\begin{picture}(2470,2529)(188,-1585)
\put(291,-88){\makebox(0,0)[lb]{\smash{{\SetFigFont{9}{10.8}{\rmdefault}{\mddefault}{\updefault}{\color[rgb]{0,0,0}$B'$}%
}}}}
\put(1325,-1530){\makebox(0,0)[lb]{\smash{{\SetFigFont{9}{10.8}{\rmdefault}{\mddefault}{\updefault}{\color[rgb]{0,0,0}$C'$}%
}}}}
\put(1320,796){\makebox(0,0)[lb]{\smash{{\SetFigFont{9}{10.8}{\rmdefault}{\mddefault}{\updefault}{\color[rgb]{0,0,0}$C$}%
}}}}
\put(860,-609){\makebox(0,0)[lb]{\smash{{\SetFigFont{9}{10.8}{\rmdefault}{\mddefault}{\updefault}{\color[rgb]{0,0,0}$A$}%
}}}}
\put(1781,-109){\makebox(0,0)[lb]{\smash{{\SetFigFont{9}{10.8}{\rmdefault}{\mddefault}{\updefault}{\color[rgb]{0,0,0}$A'$}%
}}}}
\put(2408,-627){\makebox(0,0)[lb]{\smash{{\SetFigFont{9}{10.8}{\rmdefault}{\mddefault}{\updefault}{\color[rgb]{0,0,0}$B$}%
}}}}
\end{picture}%

%% file: Fig/IdOcta.tex
\begin{picture}(0,0)%
\includegraphics{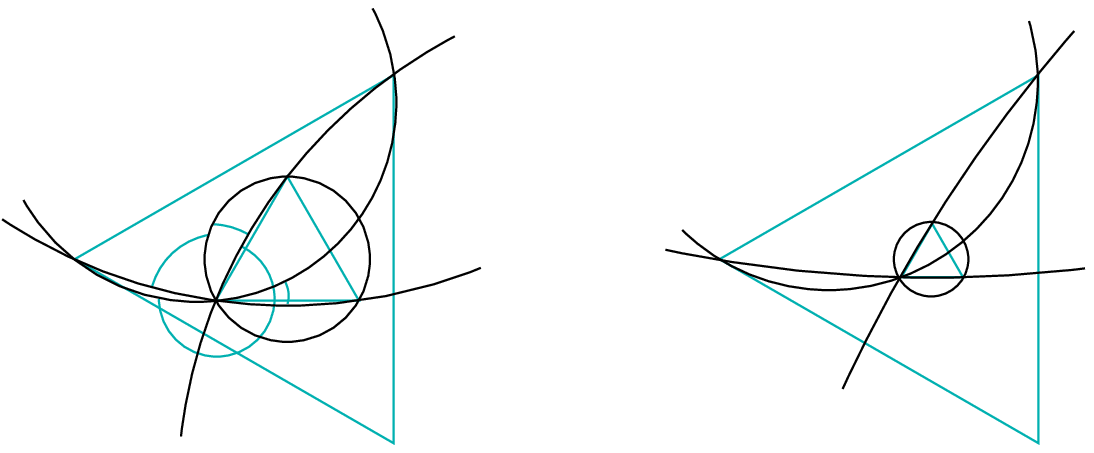}%
\end{picture}%
\setlength{\unitlength}{4972sp}%
\begingroup\makeatletter\ifx\SetFigFontNFSS\undefined%
\gdef\SetFigFontNFSS#1#2#3#4#5{%
  \reset@font\fontsize{#1}{#2pt}%
  \fontfamily{#3}\fontseries{#4}\fontshape{#5}%
  \selectfont}%
\fi\endgroup%
\begin{picture}(4141,1781)(32,-1104)
\put(161,-406){\makebox(0,0)[lb]{\smash{{\SetFigFontNFSS{10}{12.0}{\rmdefault}{\mddefault}{\updefault}{\color[rgb]{0,0,0}$C'$}%
}}}}
\put(1539,-1058){\makebox(0,0)[lb]{\smash{{\SetFigFontNFSS{10}{12.0}{\rmdefault}{\mddefault}{\updefault}{\color[rgb]{0,0,0}$A'$}%
}}}}
\put(1329,415){\makebox(0,0)[lb]{\smash{{\SetFigFontNFSS{10}{12.0}{\rmdefault}{\mddefault}{\updefault}{\color[rgb]{0,0,0}$B'$}%
}}}}
\put(1383,-570){\makebox(0,0)[lb]{\smash{{\SetFigFontNFSS{10}{12.0}{\rmdefault}{\mddefault}{\updefault}{\color[rgb]{0,0,0}$B$}%
}}}}
\put(714,-399){\makebox(0,0)[lb]{\smash{{\SetFigFontNFSS{10}{12.0}{\rmdefault}{\mddefault}{\updefault}{\color[rgb]{0,0,0}$A$}%
}}}}
\put(1212, 34){\makebox(0,0)[lb]{\smash{{\SetFigFontNFSS{10}{12.0}{\rmdefault}{\mddefault}{\updefault}{\color[rgb]{0,0,0}$C$}%
}}}}
\end{picture}%

%% file: Fig/FundGroup.tex
\begin{picture}(0,0)%
\includegraphics{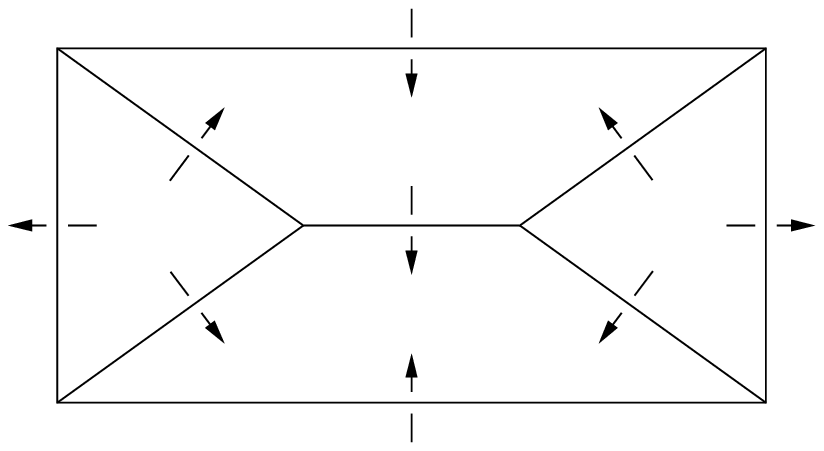}%
\end{picture}%
\setlength{\unitlength}{4144sp}%
\begingroup\makeatletter\ifx\SetFigFontNFSS\undefined%
\gdef\SetFigFontNFSS#1#2#3#4#5{%
  \reset@font\fontsize{#1}{#2pt}%
  \fontfamily{#3}\fontseries{#4}\fontshape{#5}%
  \selectfont}%
\fi\endgroup%
\begin{picture}(3780,2133)(598,-1493)
\put(2601,-1429){\makebox(0,0)[lb]{\smash{{\SetFigFontNFSS{9}{10.8}{\rmdefault}{\mddefault}{\updefault}{\color[rgb]{0,0,0}$a_3$}%
}}}}
\put(2596,493){\makebox(0,0)[lb]{\smash{{\SetFigFontNFSS{9}{10.8}{\rmdefault}{\mddefault}{\updefault}{\color[rgb]{0,0,0}$a_1$}%
}}}}
\put(613,-586){\makebox(0,0)[lb]{\smash{{\SetFigFontNFSS{9}{10.8}{\rmdefault}{\mddefault}{\updefault}{\color[rgb]{0,0,0}$a_2$}%
}}}}
\put(1536,-1039){\makebox(0,0)[lb]{\smash{{\SetFigFontNFSS{9}{10.8}{\rmdefault}{\mddefault}{\updefault}{\color[rgb]{0,0,0}$a_2a_3$}%
}}}}
\put(4314,-574){\makebox(0,0)[lb]{\smash{{\SetFigFontNFSS{9}{10.8}{\rmdefault}{\mddefault}{\updefault}{\color[rgb]{0,0,0}$a_4$}%
}}}}
\put(3204,-1047){\makebox(0,0)[lb]{\smash{{\SetFigFontNFSS{9}{10.8}{\rmdefault}{\mddefault}{\updefault}{\color[rgb]{0,0,0}$a_4a_3$}%
}}}}
\put(1544,151){\makebox(0,0)[lb]{\smash{{\SetFigFontNFSS{9}{10.8}{\rmdefault}{\mddefault}{\updefault}{\color[rgb]{0,0,0}$a_2a_1$}%
}}}}
\put(3200,146){\makebox(0,0)[lb]{\smash{{\SetFigFontNFSS{9}{10.8}{\rmdefault}{\mddefault}{\updefault}{\color[rgb]{0,0,0}$a_4a_1$}%
}}}}
\put(2259,-756){\makebox(0,0)[lb]{\smash{{\SetFigFontNFSS{9}{10.8}{\rmdefault}{\mddefault}{\updefault}{\color[rgb]{0,0,0}$a_1^{-1}a_3$}%
}}}}
\end{picture}%